\documentclass[A4,11pt]{amsart}
\usepackage{amsthm,amsmath,amssymb,eucal}
\usepackage{geometry}
\usepackage{resizegather}
\usepackage{graphicx}
\usepackage{color}
\usepackage[normalem]{ulem}
\usepackage[latin1]{inputenc}
\usepackage{mathtools}
\usepackage{url}
\mathtoolsset{showonlyrefs}

\newtheorem{claim}{Claim}[section]
\newtheorem{theorem}[claim]{Theorem}

\newtheorem{lemma}[claim]{Lemma}
\newtheorem{corollary}[claim]{Corollary}
\theoremstyle{definition}
\newtheorem{remark}[claim]{Remark}
\newtheorem{example}[claim]{Example}

\definecolor{jgreen}{rgb}{0.8, 0.8, 0.2}

\newcommand{\soutg}{\bgroup\markoverwith{\textcolor{green}{\rule[.5ex]{2pt}{1pt}}}\ULon}
\newcommand{\soutb}{\bgroup\markoverwith{\textcolor{blue}{\rule[.5ex]{2pt}{1pt}}}\ULon}
\newcommand{\soutr}{\bgroup\markoverwith{\textcolor{red}{\rule[.5ex]{2pt}{1pt}}}\ULon}

\makeatletter
\@namedef{subjclassname@2020}{%
  \textup{2020} Mathematics Subject Classification}
\makeatother

\newcommand{\so}{{\rm o}}
\newcommand{\ds}{\displaystyle}
\newcommand{\dsum}{\ds\sum}

\newcommand{\eqskip}{ \vspace*{2mm}\\ }
\newcommand{\fr}[2]{\frac{\ds #1}{\ds #2}}

\DeclareMathOperator*{\re}{Re}
\newcommand{\tto}{ \mathcal{T} }

\newcommand{\ham}{T}
\newcommand{\hamab}{\ham_{\alpha,\beta}}

\newcommand{\R}{\mathbb{R}}

\title[Spectral determinant for second order elliptic operators on the real line]{The spectral determinant for second order elliptic operators on the real line}

\author{Pedro Freitas} 
\author{Ji\v{r}\'{\i} Lipovsk\'{y}}

\address{Grupo de F\'{\i}sica Matem\'{a}tica, Instituto Superior T\'{e}cnico, Universidade de Lisboa, Av. Rovisco Pais 1,
P-1049-001 Lisboa, Portugal}
\email{pedrodefreitas@tecnico.ulisboa.pt}
\address{Department of Physics, Faculty of Science, University of Hradec Kr\'alov\'e, Rokitansk\'eho 62,
500\,03 Hradec Kr\'alov\'e, Czechia}
\email{jiri.lipovsky@uhk.cz}

\date{\today}

\begin{document}

\begin{abstract}
 We derive an expression for the spectral determinant of a second-order elliptic differential operator~$\tto$ defined on the
 whole real line, in terms of the Wronskians of two particular solutions of the equation $\tto u=0$. Examples of application of the resulting
 formula include the explicit calculation of the determinant of harmonic and anharmonic oscillators with an added bounded potential with compact support.
\end{abstract}


\maketitle

\section{Introduction}
Given an elliptic differential operator $\tto$ with discrete spectrum $\lambda_{1}\leq\lambda_{2}\leq \dots$,
we define the associated zeta function by the series
\[
 \zeta_{\tto}(s) = \dsum_{k=1}^{\infty} \fr{1}{\lambda_{k}^s},
\]
convergent on some half-plane $\re(s)>\sigma_{0}$. Zeta functions associated in this way with the Laplacian
and other elliptic operators have been the object of study for more than seventy years, dating at least
as far back as the $1949$ paper by S. Minakshisundaram and \AA. Pleijel~\cite{MP}. In the case of the Laplacian
defined on a closed manifold, they showed, among other things, that the corresponding function of a complex
variable $\zeta_{\Delta}$ may be extended as a meromorphic function to the entire complex plane and that, in particular,
it is well-defined and analytic at the origin.

One focus of attention in the study of $\zeta_{\tto}$ are the values taken by this function at specific points in
the complex plane, as well as the residues at its poles -- see~\cite{Vor80}, for instance. Here the value of $\zeta_{\tto}'(0)$
is of particular importance as it can be used to define a determinant associated with $\tto$~\cite{rasi,geya}. More precisely,
for such a class of operators, and by invoking a formal analogy with the finite-dimensional case, we may then use the expression
\[
 \det\left(\tto\right) := e^{-\zeta_{\tto}'(0)},
\]
to define the spectral (or functional) determinant of the operator $\tto$. In some cases the spectrum is known explicitly
and then expressions for the function $\zeta_{\tto}$ and the corresponding determinant may be found also in explicit form -- see~\cite{cf23}
and the references therein, for instance. However, it also happens that even when the spectrum is known explicitly it may be difficult to
find an explicit expression for the corresponding determinant~\cite{bgke,cf23,Fre18}. On the other hand, there are also examples where, although eigenvalues may
not be written explicitly in closed form, it is possible to do so for the determinant, one such case being the anharmonic oscillator which may be found in \cite{Vor80, Vor99}, and which we use in  Section 4 below~-- see also~\cite{aursal,frli20} for other examples.

In the case of operators defined on unbounded domains, and although there are some results on the half-line (see~\cite{geki19,geki19b,haleve}, for instance), to the best of our knowledge the case of Schr\"{o}dinger operators defined on the whole line has not been addressed -- see also Section 4.4 in~\cite{Vor00}, where an approach based on the combination of the results for two half-lines is considered in a particular case.
The purpose of this paper is thus to derive and prove a formula for the determinant in the case of second-order differential operators on the whole real line.
To achieve this, we first identify an appropriate setting in which the technique used in~\cite{LS77} for the case of a bounded interval may be extended to $\R$.
In fact, this turns out to be a crucial step. Apart from requiring that the spectrum be discrete with finite multiplicities, these conditions turn out to
be closely associated with the operator being in the limit-point case at $\pm \infty$ in Weyl's classification~\cite{weyl10}, in the situation where one non-trivial solution in $L^{2}\left(\R^{\pm}\right)$ exists. Potentials satisfying these conditions include, for instance, those of the form $|x|^\beta + q(x)$ for any positive real numbers $\beta$ and $q$ being a sufficiently
regular compactly supported function -- see Corollary~\ref{cor:fedosova}.

We describe the appropriate setting in detail in the next section, where we also state
our main results. The proof of the main (general) result is done in Section~\ref{sec:proofmain}, while the study of the 
perturbed anharmonic potential is presented in Section~\ref{sec:anharmonic}. In the last section we provide some explicit examples.

\section{Background and main results}

In the particular case of a Sturm-Liouville operator of the form $H_{\alpha} = -d^2/dx^2 + \alpha q$ on the interval $(0,1)$ with Dirichlet boundary conditions, for a smooth potential $q$, Levit and Smilansky~\cite{LS77}
produced the following elegant formula for the determinant, namely, $\det(H_{\alpha}) = 2y(1)$, where $y$ is the solution of the initial value
problem
\begin{equation}\label{initsl}
\left\{
\begin{array}{l}
 -y''(t) + \alpha q(t)y(t) = 0\,,\eqskip
 y(0) = 0, \; y'(0) = 1\,.
\end{array}
\right.
\end{equation}
Although Levit and Smilansky choose $\alpha \in [0,1]$, it can be chosen to be real.
In order to derive the corresponding formula for the determinant of a similar operator, now defined on the whole of the real line,
we shall adapt the argument in~\cite{LS77} to this situation. As a first step, we remark that
the above determinant for nonzero values of $\alpha$ may be related to the determinant at vanishing $\alpha$ by
\begin{equation}\label{wronskian1}
 \det(H_{\alpha}) = \fr{W(\alpha)}{W(0)} \det(H_{0}),
\end{equation}
where $W$ denotes the Wronskian associated with the solution of~\eqref{initsl} and the solution of the same differential equation satisfying the conditions $y(1)=0$,
$y'(1)=-1$ -- note that for $\alpha\neq 0$ the two solutions are either linearly independent or they are (non-zero) multiples of each other. In the latter case, both $W$ and the determinant vanish, as zero is an eigenvalue for that value of $\alpha$. For $\alpha$ equal to zero, should the two solutions be linearly dependent, then the above identity is not applicable.  In the case of the whole real line, and in order to be able to derive an analogous formula, we will need to impose certain (natural) conditions 
on the class of potentials that may be allowed. We shall thus consider operators of the form
\[
 T_{\alpha} = -\fr{d^{2}}{dx^{2}} + V(x) + \alpha q(x),
\]
where the potentials $V$ and $q$ are such that $T_{\alpha}$ is bounded from below and has discrete spectrum of the form
\begin{equation}\label{asymptotic}
 \lambda_{1} < \lambda_{2} < \dots <\lambda_{k} < \dots 
 \mbox{ with } \lambda_{k} = c k^{\tau} + \so\left(k^{\tau}\right),
\end{equation}
for some positive $\tau$. Note that the value of $\tau$ will, in general, be independent of $\alpha$ for $q$
bounded with compact support, for instance. This will be the case for potentials $V$ of the form $|x|^\beta$ for
positive real $\beta$, for instance~\cite{FN21}; see also~\cite{ms} for similar results for more general
potentials. From~\cite{hawi1} (see also~\cite{hawi2}) this implies
that not only does the operator $T_{\alpha}$ correspond to the limit-point case at $\pm \infty$,
but we also have the existence of unique (modulus multiplication by a non-zero constant factor that we choose to be $\alpha$-independent) solutions $y_{\pm}$ which are in $L^{2}(\R^{\pm})$
and, in fact, converge to zero as $t$ goes to $\pm \infty$. This is a key ingredient which will ensure the existence of solutions for two initial
value problems, allowing us to derive the determinant by means of the corresponding Wronskians in the same way as given by~\eqref{wronskian1} above.
We point out that related solutions also play a role in the context of determining exact WKB solutions, where they are sometimes referred to as {\it canonical recessive} solutions~\cite{Vor12}. Wronskians for such problems on the whole real line have also been considered in relation to Fredholm determinants, albeit under different conditions on the potential $V$~\cite{gesz86}.

In order to be able to use a method inspired by the proof given in~\cite{LS77}, we need one second ingredient, namely, that the eigenvalues $\lambda_{k}$
satisfy the asymptotic condition
\begin{equation}\label{spectr_cond}
 \fr{\lambda_k(\alpha)}{\lambda_k(0)} = 1+\mathcal{O}\left(k^{-1-\varepsilon}\right)\; \mbox{ as } k\to\infty,
\end{equation}
for some $\varepsilon>0$. This condition ensures that the resulting infinite product of the quotient above will be convergent, allowing for the comparison between $\det(\alpha)$ and $\det(0)$. As in the case of equation~\eqref{asymptotic} above, this will hold for a large class of potentials $V$ and $q$, such as
anharmonic potentials $V$ of the form $|x|^{\beta}$ for positive real $\beta$ and $q$ having
compact support~\cite{FN21}. This restriction on the perturbation potential $q$ will be used throughout the paper,
as it is required in~\cite{FN21} to ensure the condition~\eqref{spectr_cond} above, and has also been imposed in other works such as when studying trace formul\ae\ for perturbations of the harmonic oscillator in~\cite{PS06}. It also yields an additional consequence, namely, the fact that the solutions $y_{\pm}$ introduced above will be independent of $\alpha$ on the intervals $(\mathrm{sup\,} \mathrm{supp\,}(q),\infty)$ and
$(-\infty,\mathrm{inf\,} \mathrm{supp\,}(q))$, respectively. We note, however, that there are results on related issues 
with different restrictions on the potentials $q$, ranging from imposing a certain rate of decay at infinity~\cite{abp} to $q$ being in $L^{1}\left(\R\right)$~\cite{frke16}.
We thus believe that the condition that $q$ has compact support is of a more technical nature, but at this point we need it in order to obtain our results.

We are now ready to state our main result for general potentials $V$ and a corresponding perturbation $q$ ensuring the above conditions~\eqref{asymptotic} and~\eqref{spectr_cond}.
\begin{theorem}\label{thm:maingen} Let $T_{\alpha}$ be the Schr\"{o}dinger operator defined on the real line by
\[
 T_{\alpha} = -\fr{d^{2}}{d x^2} + V(x) + \alpha q(x),
\]
where $q$ is bounded with compact support $\mathrm{supp\,}(q)\subset[a,b]$, $\alpha \in \mathbb{R}$, and $V$ is such that $T_{\alpha}$ has
discrete spectrum satisfying conditions~\eqref{asymptotic} and~\eqref{spectr_cond}. Then, the spectral determinant
for the operator $T_{\alpha}$ on $\mathbb{R}$ is well-defined and, assuming that $\det(T_{0})$ does not vanish, it satisfies
$$
  \mathrm{det}\left(T_{\alpha}\right) = \fr{W(\alpha)}{W(0)} \mathrm{det}\left(T_{0}\right)\,,
$$
where $W(\alpha)$ is the Wronskian of the solutions $y_{\pm}$ belonging to $L^{2}\left(\R^{\pm}\right)$.
\end{theorem}
\begin{remark}
 It is possible to consider more general operators with differential part given by $-\frac{\mathrm{d}}{\mathrm{d}x}\left(p
\frac{\mathrm{d}}{\mathrm{d}x}\right)$, for suitable (positive) functions $p$ for which the given conditions are still satisfied.
\end{remark}

\begin{remark}
 The condition that $\det(T_{0})$ does not vanish has to be imposed, as in that case the two solutions $y_{\pm}$ are, in
 fact, linearly dependent, and zero is an eigenvalue of $T_{0}$.
\end{remark}

We note that the conditions indicated include important families of potentials $V$ and $q$, such as bounded perturbations with compact support of
anharmonic operators, for which we can ensure that the conditions in Theorem~\ref{thm:maingen} hold.
\begin{corollary}\label{cor:fedosova}
 Let $V$ be the anharmonic potential defined on the real line by $V(x)= \left|x\right|^{\beta}$, for some positive real number $\beta$, and $q$ be
 a compactly supported piecewise H\"{o}lder continuous function with positive constant. Then the potential $V + \alpha q$ satisfies the conditions
 in Theorem~\ref{thm:maingen}.
\end{corollary}
\begin{proof}
 This follows from the results in~\cite{FN21} and, in particular, from the asymptotics
 \[
 \lambda_{n} = c_{1} (2n-1)^{2\beta/(\beta+2)} + c_{2} (2n-1)^{-2/(\beta+2)} + \so\left({n^{-2/(\beta+2)}}\right),
 \]
 as $n$ goes to infinity, where $c_{1}$ does not depend on $q$.
\end{proof}
\noindent This case will be considered in more detail in Section~\ref{sec:anharmonic} where, in particular, we will make the solutions $y_{\pm}$ explicit, with the examples provided in the last section then showing that it is possible to obtain explicit formul\ae\ for the determinants of this type of operators.

\section{Proof of Theorem~\ref{thm:maingen}\label{sec:proofmain}}

In this section we will prove Theorem~\ref{thm:maingen} following the arguments of~\cite{LS77}, now adapted to the setting of the whole real line. We start with a lemma on the relation between the Wronskian and the Green's function. 

\begin{lemma}\label{lem:dW=W int}
Let $y_\pm$ be the unique (modulus multiplication by a nonzero constant) nontrivial solutions of $T_\alpha y =0$ in
$L^{2}\left(\mathbb{R}_\pm\right)$, respectively. Let $W(\alpha) = y_-'(x) y_+(x)- y_-(x) y_+'(x)$ be the Wronskian of the solutions $y_-$ and $y_+$. Let $G(x,s)$ be the Green's function for the operator $T_\alpha$. Under conditions \eqref{asymptotic} and \eqref{spectr_cond} we have 
$$
  \frac{1}{W(\alpha)}\frac{\partial W(\alpha)}{\partial \alpha} = \int_{-\infty}^\infty q(x) G(x,x) \,\mathrm{d}x
$$
\end{lemma}
\begin{proof}
The functions $y_-$ and $y_+$ satisfy the equations
\begin{equation}
  -y_-''(x)+ V(x) y_-(x)+\alpha q(x) y_-(x) = 0\,\label{eq:y1}
\end{equation}
and
\begin{equation}
  -y_+''(x)+ V(x) y_+(x)+\alpha q(x) y_+(x) = 0\,.\label{eq:y2}  
\end{equation}
On $[a,b]$ we define their $\alpha$-derivatives $z_-(x) = \frac{\partial y_-}{\partial \alpha}(x)$ and $z_+(x) =
\frac{\partial y_+}{\partial \alpha}(x)$. Since $y_\pm$ are $\alpha$-independent on $(b,\infty)$ and $(-\infty,a)$,
respectively, one can find by differentiating \eqref{eq:y1} and \eqref{eq:y2} that they satisfy the equations
\begin{align}
  -z_-''(x)+ V(x) z_-(x)+\alpha q(x) z_-(x) = -q(x)y_-(x)\,,\label{eq:z1}\eqskip
  -z_+''(x)+ V(x) z_+(x)+\alpha q(x) z_+(x) = -q(x)y_+(x)\,,\label{eq:z2}  
\end{align}
subject to the initial conditions
\begin{equation}
  z_-(a) = z_-'(a) = 0\,,\quad z_+(b) = z_+'(b) = 0\,.\label{eq:bcz}
\end{equation}
Let $c$ be any point in $[a,b]$. Now we multiply \eqref{eq:y1} by $z_+$, subtract \eqref{eq:z2} multiplied by $y_-$ and we integrate the whole expression from $c$ to $b$. We add to this integral the integral from $a$ to $c$ from \eqref{eq:y2} multiplied by $z_-$ minus \eqref{eq:z1} multiplied by $y_+$. We obtain
\begin{align*}
  \int_c^b (-z_+(x)y_-''(x)+y_-(x)z_+''(x))\,\mathrm{d}x+ \int_a^c (- z_-(x)y_+''(x)+y_+(x)z_-''(x))\,\mathrm{d}x =\eqskip
= \int_a^b q(x)y_-(x)y_+(x)\,\mathrm{d}x
\end{align*}
Integrating by parts we get
\begin{align}
 -z_+(b)y_-'(b)+y_-(b)z_+'(b)+z_-(a)y_+'(a)-y_+(a)z_-'(a)-z_-(c)y_+'(c)\nonumber\eqskip
+y_+(c)z_-'(c)+z_+(c)y_-'(c)-y_-(c)z_+'(c) = \int_a^b q(x)y_-(x)y_+(x)\,\mathrm{d}x\,. \label{eq:afterintbyparts}
\end{align}
The first four terms on the \emph{lhs} of \eqref{eq:afterintbyparts} vanish due to \eqref{eq:bcz} and we obtain the $\alpha$-derivative of the Wronskian on the \emph{lhs}.
\begin{equation}
\left.\frac{\partial}{\partial \alpha} (y_+y_-'-y_-y_+')\right|_c = \int_a^b q(x)y_-(x)y_+(x)\,\mathrm{d}x\,.\label{eq:derwronskian}
\end{equation}
If $y_{\pm}$ are linearly independent, then the Green's function of a Hamiltonian $T_{\alpha}$ satisfying condition~\eqref{asymptotic} is given by
$$
  G(x,s) = \left\{\begin{matrix}\fr{y_-(x)y_+(s)}{W(\alpha)}\,,& \mathrm{for}\quad x<s\,,\\  \fr{y_-(s)y_+(x)}{W(\alpha)}\,,& \mathrm{for}\quad s<x\,.\end{matrix}\right.
$$
(see, e.g., \cite[eq. (2.18.4)]{Tit62}). In particular, we have
\begin{equation}
  G(x,x) = \fr{y_-(x) y_+(x)}{y_-'(x) y_+(x)- y_-(x) y_+'(x)}\,. \label{eq:greenxx}
\end{equation}
Using \eqref{eq:derwronskian}, \eqref{eq:greenxx} and the fact that $\mathrm{supp}\,q\subset [a,b]$, the result follows.
\end{proof}

Let us define 
\begin{equation}
  f(\alpha) = \prod_{j=1}^\infty \frac{\lambda_j(\alpha)}{\lambda_j(0)}\,,\label{eq:prod}
\end{equation}
where $\lambda_j(\alpha)$ is the $j$-th eigenvalue of the operator $T_\alpha$. 
Condition~\eqref{spectr_cond} implies that the product above is uniformly convergent and well-defined. Differentiating with respect to $\alpha$ yields
\begin{equation}
  \frac{1}{f(\alpha)}\frac{\mathrm{d} f(\alpha)}{\mathrm{d} \alpha} = \sum_{j=1}^\infty \frac{1}{\lambda_j(\alpha)}\frac{\mathrm{d}\lambda_j(\alpha)}{\mathrm{d}\alpha}\,.\label{eq:f}
\end{equation}

Let $u_j(x,\alpha)$ be a $j$-th eigenfunction of $T_\alpha$, normalised to have unit $L^{2}(\mathbb{R})$ norm, i.e.
\begin{align}
  T_\alpha u_j(x,\alpha) = \lambda_j(\alpha) u_j(x,\alpha)\,,\label{eq:eigeq}\eqskip
  \int_{-\infty}^\infty |u_j(x,\alpha)|^2\,\mathrm{d}x = 1\,.\label{eq:normu} 
\end{align}

The aim of the following lemma is to relate the function $f(\alpha)$ to the Wronskian.

\begin{lemma}
There exists an $\alpha$-independent constant $C$ such that $C\, f(\alpha) =W(\alpha)$.
\end{lemma}
\begin{proof}
We rewrite the eq.~\eqref{eq:f} as
$$
  \frac{1}{f(\alpha)}\frac{\mathrm{d} f(\alpha)}{\mathrm{d} \alpha} = \int_{-\infty}^\infty \sum_{j=1}^\infty  \frac{|u_j(x,\alpha)|^2}{\lambda_j(\alpha)} q(x)\,\mathrm{d}x = \int_{-\infty}^\infty G(x,x) q(x)\,\mathrm{d}x = \frac{1}{W(\alpha)}\frac{\mathrm{d} W(\alpha)}{\mathrm{d} \alpha}\,,
$$
where we have used the first Hellmann-Feynman theorem (first proven by G{\"u}ttinger \cite{Gut32})
$$
  \frac{\mathrm{d}\lambda_j(\alpha)}{\mathrm{d}\alpha} = \int_{-\infty}^{+\infty} |u_j(x,\alpha)|^2 q(x)\,\mathrm{d}x  = \int_a^b |u_j(x,\alpha)|^2 q(x)\,\mathrm{d}x\,,
$$
the formula
\begin{equation}
  G(x,s) = \sum_{j=1}^\infty  \frac{\bar{u}_j(x,\alpha)u_j(s,\alpha)}{\lambda_j(\alpha)}\,, \label{eq:green2}
\end{equation}
Lemma~\ref{lem:dW=W int} and the fact that the sum in \eqref{eq:f} is absolutely convergent and hence we can exchange the sum and the integral. Hence
\begin{equation}
  \frac{\mathrm{d}}{\mathrm{d}\alpha}\mathrm{log\,}(f(\alpha)) = \frac{\mathrm{d}}{\mathrm{d}\alpha}\mathrm{log\,}(W(\alpha)) \label{eq:log}
\end{equation}
from which the claim follows. 
\end{proof}

From this lemma, the main theorem follows.

\section{The anharmonic oscillator\label{sec:anharmonic}}

Let us now consider a special case of the operator $T_\alpha$ for the potential $V(x) = |x|^\beta$, the operator $\hamab = -
\frac{\mathrm{d}^2}{\mathrm{d}x^2}+|x|^\beta+\alpha q(x)$ on $\mathbb{R}$ with $q$ being a real-valued, bounded,
compactly supported H\"older continuous function with the exponent larger than 0, $\mathrm{supp\,}q \subset [a,b]$ with $-\infty<a\leq 0\leq b<\infty$
and $\beta>0$. Due to Corollary~\ref{cor:fedosova} it satisfies the condition in Theorem~\ref{thm:maingen}.

Let $\beta$ be fixed and $ y_-(x)$ and $ y_+(x)$ be the solutions of the equation $\hamab y = 0$ on $\mathbb{R}$ that behave as
\begin{equation}
   y_-(x) = \sqrt{-x} K_{\frac{1}{2+\beta}}\left(\frac{2}{2+\beta}(-x)^{1+\frac{\beta}{2}}\right)\quad \mathrm{for}\ x<a \label{eq:tildey1}
\end{equation}
\begin{equation}
   y_+(x) = \sqrt{x} K_{\frac{1}{2+\beta}}\left(\frac{2}{2+\beta}x^{1+\frac{\beta}{2}}\right)\quad \mathrm{for}\ x>b \label{eq:tildey2}
\end{equation}
respectively. Here $K_{\nu}(z)$ is the Bessel $K$ function. We do not denote the dependence of the functions $ y_-$, $ y_+$ on $\alpha$
and $\beta$ explicitly. The fact that the above functions are solutions to the equation $\hamab y = 0$ can be verified by a straightforward computation
using~\cite[\S 17.7 and \S 17.71]{WW96}. Using the formul\ae\ $K_{-\nu}(z) = K_{\nu}(z)$ and $\frac{\partial}{\partial z}(z^\nu K_\nu(z)) = -z^\nu K_{\nu-1}(z)$ (see,
e.g.,~\cite[51:5:1 and 51:10:4]{SO87}) one can straightforwardly find that the derivatives of~\eqref{eq:tildey1} and~\eqref{eq:tildey2} are
\begin{align*}
   y_-'(x) &= (-x)^{\frac{\beta+1}{2}} K_{\frac{\beta+1}{\beta+2}}\left(\frac{2}{2+\beta}(-x)^{1+\frac{\beta}{2}}\right)\quad \mathrm{for}\ x<a\eqskip
   y_+'(x) &=   -x^{\frac{\beta+1}{2}} K_{\frac{\beta+1}{\beta+2}}\left(\frac{2}{2+\beta}x^{1+\frac{\beta}{2}}\right)\quad \mathrm{for}\ x>b 
\end{align*}
The above-mentioned functions fulfill the conditions
\begin{align}
   y_-(a) &= \sqrt{-a}\, K_{\frac{1}{2+\beta}}\left(\frac{2}{2+\beta}(-a)^{1+\frac{\beta}{2}}\right) \,,\label{eq:bc1}\eqskip
   y_-'(a) &=  (-a)^{\frac{\beta+1}{2}} K_{\frac{\beta+1}{\beta+2}}\left(\frac{2}{2+\beta}(-a)^{1+\frac{\beta}{2}}\right)\,,\label{eq:bc2}\eqskip
   y_+(b)  &= \sqrt{b}\, K_{\frac{1}{2+\beta}}\left(\frac{2}{2+\beta}b^{1+\frac{\beta}{2}}\right)\,,\label{eq:bc3}\eqskip
   y_+'(b) &=	  -b^{\frac{\beta+1}{2}} K_{\frac{\beta+1}{\beta+2}}\left(\frac{2}{2+\beta}b^{1+\frac{\beta}{2}}\right)\,.\label{eq:bc4}
\end{align}

We can rewrite Theorem~\ref{thm:maingen} into the setting that uses only the boundary value problem on a finite interval.

\begin{theorem}\label{thm:main}
The spectral determinant for the operator $\hamab$ on $\mathbb{R}$ can be expressed as 
$$
  \mathrm{det}\,\hamab = \frac{W(\alpha)}{W(0)} \mathrm{det}\,\ham_{0,\beta}\,,
$$
where $W(\alpha)$ is the Wronskian of the solution $y_-$ to the Cauchy problem $\hamab y_{-} = 0$ with the initial conditions~\eqref{eq:bc1}
and~\eqref{eq:bc2} and the solution $y_+$ to the Cauchy problem $\hamab y_+ = 0$ with the initial conditions~\eqref{eq:bc3} and~\eqref{eq:bc4},
similarly, $W(0)$ is the above Wronskian for the parameter $\alpha = 0$ and $\ham_{0,\beta}$ is the Hamiltonian for the parameter $\alpha =0$.
\end{theorem}

For $\beta = 2$ the spectral determinant can be found explicitly. We define $ v_-(x)$ and $ v_+(x)$ as the solutions to the equation $\ham_{\alpha,2} v = 0$ on $\mathbb{R}$ that behave as 
\begin{equation}
    v_-(x) = \frac{1}{2^{1/4}} D_{-\frac{1}{2}}(-\sqrt{2}x)  = \mathrm{e}^{-\frac{x^2}{2}}H_{-\frac{1}{2}}(-x)\quad \mathrm{for}\ x<a \label{eq:v1}
\end{equation}
and 
\begin{equation}
   v_+(x) = \frac{1}{2^{1/4}} D_{-\frac{1}{2}}(\sqrt{2}x) = \mathrm{e}^{-\frac{x^2}{2}}H_{-\frac{1}{2}}(x) \quad \mathrm{for}\ x>b\,, \label{eq:v2}
\end{equation}
respectively. Here $D_{\nu}(z)$ is the parabolic cylinder function and $H_{\nu}(x)$ Hermite function. These functions are $\frac{1}{\sqrt{2\pi}}$-multiples of
the above-mentioned functions $y_-$ and $y_+$ for $\beta=2$ \cite[46:4:1 and 46:4:5]{SO87}. The derivatives of $v_-$ and $v_+$ can be computed with the use of the formula 
\begin{equation}
  \frac{\mathrm{d}}{\mathrm{d}z}D_\nu(z) = \frac{z}{2}D_\nu(z)-D_{\nu+1}(z)\label{eq:parabolic:der}
\end{equation} 
(see, e.g., \cite[46:10:1]{SO87} and they are equal to
\begin{align}
   v_-'(x) &= \frac{1}{2^{\frac{1}{4}}}[x D_{-\frac{1}{2}}(-\sqrt{2}x)+\sqrt{2}D_{\frac{1}{2}}(-\sqrt{2}x)] \quad \mathrm{for} \ x<a\,,\label{eq:v1p}\eqskip
   v_+'(x) &= \frac{1}{2^{\frac{1}{4}}}[x D_{-\frac{1}{2}}(\sqrt{2}x)-\sqrt{2}D_{\frac{1}{2}}(\sqrt{2}x)] \quad \mathrm{for}\ x>b\,,\label{eq:v2p}
\end{align}

It can be straightforwardly found that they fulfill the conditions
\begin{align}
   v_-(a) &=  \frac{1}{2^{1/4}} D_{-\frac{1}{2}}(-\sqrt{2}a) \,,\label{eq:bc1v}\eqskip
   v_-'(a) &= \frac{1}{2^{1/4}} a D_{-\frac{1}{2}}(-\sqrt{2}a)+2^{1/4} D_{\frac{1}{2}}(-\sqrt{2}a) \,,\label{eq:bc2v}\eqskip
   v_+(b) &= \frac{1}{2^{1/4}} D_{-\frac{1}{2}}(\sqrt{2}b)\,,\label{eq:bc3v}\eqskip
   v_+'(b) &= \frac{1}{2^{1/4}} b D_{-\frac{1}{2}}(\sqrt{2}b)-2^{1/4} D_{\frac{1}{2}}(\sqrt{2}b)\,.\label{eq:bc4v}
\end{align}

Then Theorem~\ref{thm:main} has the following corollary.
\begin{corollary}\label{cor:main}
The spectral determinant for the operator $\ham_{\alpha,2}$ on $\mathbb{R}$ is equal to the Wronskian $v_-'(x) v_+(x)- v_-(x) v_+'(x)$ of the solution $v_-$ to the
Cauchy problem $\ham_{\alpha,2} v_- = 0$ with the initial conditions \eqref{eq:bc1v} and \eqref{eq:bc2v} and the solution $v_+$ to the Cauchy problem  $\ham_{\alpha,2} v_+ = 0$
with the initial conditions \eqref{eq:bc3v} and \eqref{eq:bc4v}.
\end{corollary}
\begin{proof}
The functions $ v_-$ and $ v_+$ are $\frac{1}{\sqrt{2\pi}}$-multiples of the functions $ y_-$ and $ y_+$ for $\beta=2$, respectively. For $\alpha = 0$
the functions $ v_-$ and $ v_+$ (and hence $v_-$ and $v_+$) have due to the formula $D_\nu(0) = \frac{\sqrt{2^\nu \pi}}{\Gamma(\frac{1-\nu}{2})}$
(see, e.g.,~\cite[46:7]{SO87}) the values
\begin{align}
   v_-(0) =  v_+(0) &= \frac{\sqrt{\pi}}{\sqrt{2}\Gamma(\frac{3}{4})}\,,\label{eq:v10}\eqskip
   v_-'(0) = - v_+'(0) &= \frac{\sqrt{2\pi}}{\Gamma(\frac{1}{4})}\,.\label{eq:v10p}  
\end{align}
The Wronskian of $v_-$ and $v_+$ in the case of $\alpha=0$ is 
$$
   W(0) = v_-'(x) v_+(x)- v_-(x) v_+'(x) = \frac{2\pi}{\Gamma(\frac{1}{4})\Gamma(\frac{3}{4})} = \sqrt{2}\,,
$$
where we have used the formula $\Gamma(\frac{3}{4})\Gamma(\frac{1}{4}) = \pi\sqrt{2}$ (see, e.g., \cite[43:4:14]{SO87}). This together with Theorem~\ref{thm:main} and
the fact that the spectral determinant for the harmonic oscillator on $\mathbb{R}$ is equal to $\sqrt{2}$ (see \cite{Fre18}) concludes the proof.
\end{proof}

Moreover, one can obtain the spectral determinant for general positive integer $\beta$ as another corollary of Theorem~\ref{thm:main}.

\begin{corollary}\label{cor:integer}
Let $\beta$ be a positive integer. Then the spectral determinant for the operator $\ham_{\alpha,\beta}$ is equal to the Wronskian $v'_-(x)v_+(x)-v_-(x)v'_+(x)$ of
the solutions $v_\pm$ to the Cauchy problem $\ham_{\alpha,\beta} v_\pm = 0$ that are equal to $\frac{1}{\sqrt{\pi \left(1+\frac{\beta}{2}\right)}}$-multiples of the functions
$y_\pm$ defined by \eqref{eq:tildey1} and \eqref{eq:tildey2}.
\end{corollary}
\begin{proof}
For $\alpha= 0$, the limiting values of $y_+(0+)$ and $y'_+(0+)$ can be expressed using the approximate formula 
\begin{equation}
   K_\nu(z) \approx \frac{\Gamma(\nu) z^{-\nu}}{2^{1-\nu}}+\frac{\Gamma(-\nu)z^{\nu}}{2^{1+\nu}} \label{eq:51:9:2}
\end{equation} 
for small $z$ (see, e.g., \cite[51:9:2]{SO87}). Using the above-mentioned expansion for~\eqref{eq:bc3} and~\eqref{eq:bc4}, we get
\begin{align}
   y_+(b) &= \frac{1}{2}\Gamma\left(\fr{1}{\beta+2}\right)(\beta+2)^{\frac{1}{\beta+2}}+O(b)\,,\quad \beta>-1\,, \alpha=0 \label{eq:integer:y+}\eqskip
   y'_+(b) &= -\frac{1}{2}\Gamma\left(\frac{\beta+1}{\beta+2}\right)(\beta+2)^{\frac{\beta+1}{\beta+2}}+O(b^{\beta+1})\,,\quad \beta>-1\,, \alpha=0\,.\label{eq:integer:y+prime}
\end{align}
The Wronskian of $y_-$ and $y_+$ at zero and $\alpha=0$, may be obtained using the symmetry relations $y_-(0) = y_+(0)$, $y'_-(0) = -y'_+(0)$, and equals
\begin{align}
  W(0) &= -2y_+(0)y'_+(0)\nonumber\eqskip
   &= \fr{\beta+2}{2}\Gamma\left(\frac{1}{\beta+2}\right)\Gamma\left(1-\fr{1}{\beta+2}\right)\nonumber\eqskip
   &= \fr{1}{2}\pi (\beta+2)\frac{1}{\sin\left(\frac{\pi}{\beta+2}\right)}\,,\label{eq:wronskian:integer}
\end{align}
where we have used the formula 
\begin{equation}
  \Gamma(1+x)\Gamma(-x) = \frac{\pi}{\sin{(-\pi x)}} \label{eq:43:5:1}
\end{equation} 
(see, e.g., \cite[43:5:1]{SO87}).

For positive integer $\beta$ and $\alpha=0$, the determinant is given by the expression 
\begin{equation}
 \mathrm{det}(T_{0,\beta}) =\frac{1}{\sin{\left(\fr{\pi}{\beta+2}\right)}} \label{eq:integer:det0}
\end{equation} 
(see~\cite{Vor80} and~\cite{Vor99} for even and odd $\beta$, respectively). This together with~\eqref{eq:wronskian:integer} and Theorem~\ref{thm:main} yields the result.
\end{proof}

\begin{remark}
Our results can be generalized also to cases when $a>0$ or $b<0$. To do so, the solutions $ y_-$ or $ y_+$ should be prolonged to
$\mathbb{R}_+$ or $\mathbb{R_-}$, respectively, using the $\sqrt{|x|}$-multiple of a linear combination of the
modified Bessel $I$ and $K$ functions so that the function is continuous at zero with continuous derivative at zero. For the
sake of brevity of the paper we do not carry out this construction here.
\end{remark}

\section{Examples} We shall now provide some examples illustrating different possible applications of Theorem~\ref{thm:maingen}. The simplest, if somewhat more artificial, are those where the potential $q$ contains
a term that cancels the part of $V$ on the support of $q$. This makes the
initial boundary value problems in Theorem~\ref{thm:main} simpler to solve, and is the subject of Example~\ref{ex:x2} where we then analyse the determinant as a function of the length of the support of $q$. Next
we consider a step-like perturbation of the harmonic oscillator (Example~\ref{ex:x2:step}), and
finally illustrate a perturbation of the $x^4$ potential in Example~\ref{ex:x4}.

\begin{example}\label{ex:x2}
In the first example, we consider $\beta=2$ and the perturbation $q=-x^2$ with $\alpha = 1$ on the interval $[0,b]$ for a given positive real
parameter $b$. Clearly, the operator $\ham_{1,2}$ acts as the negative second derivative in the interval $[0,b]$, so the function $ v_-$ on $(-\infty,0]$
is given by \eqref{eq:v1}, the function $ v_+$ on $[b,\infty)$ is given by \eqref{eq:v2} and on $[0,b]$ as $v_+ (x) = Ax +B$ with the condition of
continuity of the function value and the derivative at $b$. For their derivatives one can use the expressions \eqref{eq:v1p} and \eqref{eq:v2p}. We
have \eqref{eq:bc3v}, \eqref{eq:bc4v}; the functions $v_-(0+)$ and $v_-'(0+)$ are equal to \eqref{eq:v10} and \eqref{eq:v10p}, respectively. Moreover,
\begin{align*}
 v_+(b-) &= Ab + B\,,\eqskip
 v_+'(b-) &= A\,,\eqskip
 v_+(0+) &= B = v_+(b-)-b v_+'(b-)\,,\eqskip
 v_+'(0+) &= A = v_+'(b-)\,.
\end{align*}
Finally, the Wronskian of $v_-$ and $v_+$ for $\alpha = 1$ is
\begin{align*}
  W(1) &= v_-'(0+) v_+(0+)- v_-(0+) v_+'(0+)
\eqskip
  &= \frac{2^{\frac{1}{4}}\sqrt{\pi}}{\Gamma\left(\frac{1}{4}\right)}[-b^2 D_{-\frac{1}{2}}\left(\sqrt{2}b\right)+\sqrt{2}b
  D_{\frac{1}{2}}\left(\sqrt{2}b\right)+D_{-\frac{1}{2}}\left(\sqrt{2}b\right)]
\eqskip
 & -\frac{\sqrt{\pi}}{2^{\frac{3}{4}}\Gamma\left(\frac{3}{4}\right)}[b D_{-\frac{1}{2}}\left(\sqrt{2}b\right)-
 \sqrt{2}D_{\frac{1}{2}}\left(\sqrt{2}b\right)]  \,.
\end{align*}
This is also by Corollary~\ref{cor:main} the spectral determinant of the operator $\ham_{1,2}(b)$. In Figure~\ref{ex:x2:fig1} we plot the dependence of the
spectral determinant on the parameter $b$. We can see that its value at $b=0$ is equal to $\sqrt{2}$, the determinant of the harmonic oscillator on
$\mathbb{R}$~\cite{Fre18}. To be more precise, we have the asymptotic behaviour $W(1)\approx \sqrt{2}+ \mathcal{O}(b^3)$ for $b\ll 1$, where we have used the formula
$$
  D_\nu(z) \approx \sqrt{2^\nu \pi }\left[\frac{1-(\frac{1}{4}+\frac{\nu}{2})z^2}{\Gamma\left(\frac{1-\nu}{2}\right)}-\frac{z\sqrt{2}}{\Gamma
  \left(-\frac{\nu}{2}\right)}\right]\quad \mathrm{for\ }z\ll 1
$$ 
(see, e.g.,~\cite[46:9:1]{SO87}). The asymptotic behaviour for large $b$ is 
$$
  W(1)\approx \frac{1}{\Gamma\left(\frac{1}{4}\right)}\mathrm{e}^{-\frac{b^2}{2}}b^{\frac{3}{2}}\sqrt{\pi}\left(1+\mathcal{O}\left(b^{-1}\right)\right)\,,
$$
where we have used the formula
$$
 D_\nu(z) \approx z^\nu \mathrm{e}^{-\frac{z^2}{4}}\left(1+\mathcal{O}\left(z^{-2}\right)\right) \mbox{ for } z \gg 1
$$
from \cite[46:9:2]{SO87}.

\begin{figure}
\centering
\includegraphics[height=5cm]{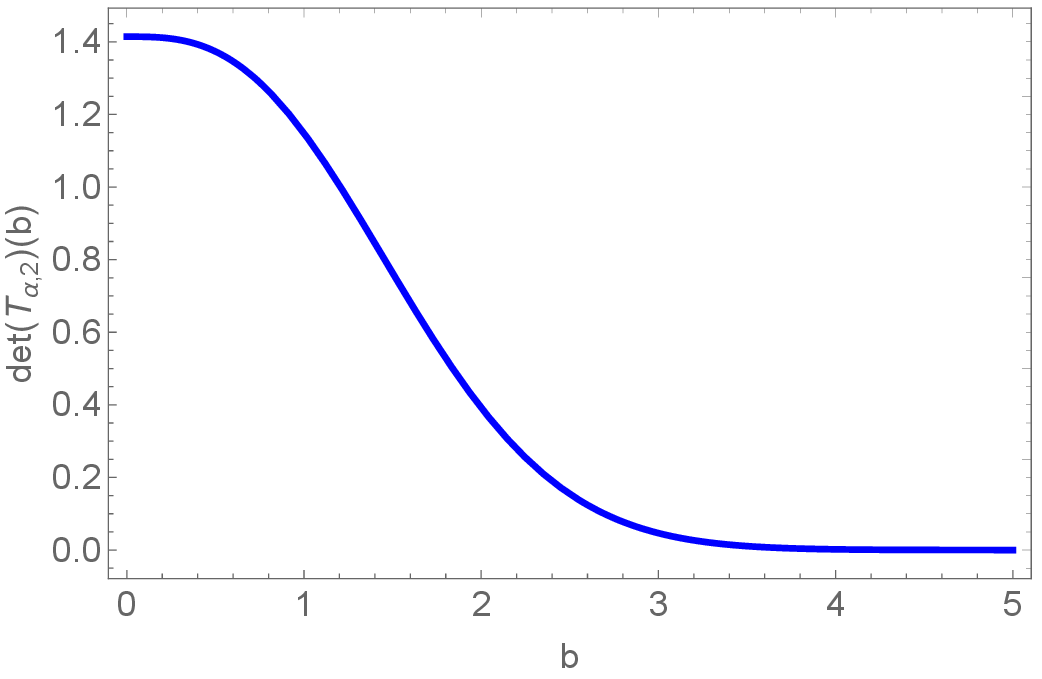}
\caption{Example~\ref{ex:x2}: dependence of the spectral determinant on the parameter $b$.}
\label{ex:x2:fig1}
\end{figure}
\end{example}

\begin{example}\label{ex:x2:step}
In the next example, we again study the case $\beta=2$, this time we choose the step-like potential $q(x) = \Theta(1-x) \Theta(x+1)$, where $\Theta$ is the Heaviside theta function.
Since the problem is symmetric with respect to $x=0$, we will focus on the function $v_+$. For $x>1$, it is given by the expression~\eqref{eq:v2} and its derivative
by~\eqref{eq:v2p}. In the interval $(-1,1)$, the function $v_+$ is the solution to the differential equation $-v''+(x^2+\alpha) v = 0$, which is given as a linear combination of the
functions
$$
  v_+(x) = c_1 D_{-\frac{1}{2}(1+\alpha)}(\sqrt{2}x)+c_2 D_{\frac{1}{2}(\alpha-1)}(i \sqrt{2}x)\,.
$$
This can be verified using the formula $\frac{\mathrm{d}^2}{\mathrm{d}z^2} D_\nu(z) = \left(\frac{z^2}{4}-\frac{1}{2}-\nu\right)D_{\nu}(z)$ (see, e.g. \cite[46:10:2]{SO87}). Due
to~\eqref{eq:parabolic:der}, the derivative on $(-1,1)$ is
\begin{align}
  v_+'(x) &= c_1\left[-\sqrt{2} D_{\frac{1}{2}(1-\alpha)}(\sqrt{2}x)+x D_{-\frac{1}{2}(1+\alpha)}(\sqrt{2}x)\right]\nonumber\eqskip
  & \hspace*{5mm} + c_2\left[-i\sqrt{2}D_{\frac{1}{2}(1+\alpha)}(i\sqrt{2}x)-x D_{\frac{1}{2}(\alpha-1)}(i\sqrt{2}x)\right]\,.\label{eq:x2step:v+prime}
\end{align}
The procedure to obtain the function value and derivative at $x=0$ is straightforward but cumbersome; ``sewing'' the function value and the derivative at $x=1$ we find the constants
$c_1$ and $c_2$ and then substitute $x=0$ to \eqref{eq:x2step:v+prime}. The Wronskian of the solutions $v_+$ and $v_-$ can be expressed using the symmetry
$v_+(0) = v_-(0)$, $v_+'(0) = -v_-'(0)$ as $W(\alpha) = -2 v_+(0) v_+'(0)$. The choice of the constants multiplying $v_\pm$ in Section~\ref{sec:anharmonic} assures that the Wronskian
for $\alpha=0$ is equal to the determinant for $\alpha=0$ and hence the determinant for general $\alpha$ is equal to the Wronskian.  We used Mathematica to compute it and found
\begin{align}\label{det1}
  \mathrm{det\,}T_{\alpha,2} &= \frac{i}{\pi}2^{1-\frac{\alpha}{2}}\cos\left(\frac{\pi\alpha}{2}\right)\frac{k_1(\alpha) k_2(\alpha)}{k_3^2(\alpha)}
\end{align}
with 
\begin{gather*}
 \begin{array}{lll}
  k_1(\alpha) &= 2^{\frac{\alpha}{2}}\Gamma\left(\fr{3+\alpha}{4}\right)D_{\frac{1}{2}}\left(\sqrt{2}\right)D_{-\frac{1}{2}(1+\alpha)}\left(\sqrt{2}\right)-2^{\frac{\alpha}{2}}
  \Gamma\left(\fr{3+\alpha}{4}\right)D_{-\frac{1}{2}}\left(\sqrt{2}\right)D_{\frac{1}{2}(1-\alpha)}\left(\sqrt{2}\right)\eqskip
& \hspace*{5mm} +\Gamma\left(\fr{3-\alpha}{4}\right)\left[\left(\sqrt{2}D_{-\frac{1}{2}}(\sqrt{2})-D_{\frac{1}{2}}(\sqrt{2})\right)
D_{\frac{1}{2}(\alpha-1)}(i\sqrt{2})+iD_{-\frac{1}{2}}(\sqrt{2})D_{\frac{1}{2}(1+\alpha)}(i\sqrt{2})\right]\,,\eqskip
  k_2(\alpha) &= 2^{\frac{\alpha}{2}}\Gamma(\frac{1+\alpha}{4})D_{\frac{1}{2}}(\sqrt{2})D_{-\frac{1}{2}(1+\alpha)}(\sqrt{2})
  -2^{\frac{\alpha}{2}}\Gamma(\frac{1+\alpha}{4})D_{-\frac{1}{2}}(\sqrt{2})D_{\frac{1}{2}(1-\alpha)}(\sqrt{2})\eqskip
 &-i\Gamma(\frac{1-\alpha}{4})D_{\frac{1}{2}(-1+\alpha)}(i\sqrt{2})[\sqrt{2}D_{-\frac{1}{2}}(\sqrt{2})-D_{\frac{1}{2}}(\sqrt{2})]
 +\Gamma(\frac{1-\alpha}{4})D_{-\frac{1}{2}}(\sqrt{2})D_{\frac{1}{2}(1+\alpha)}(i\sqrt{2})\,,\eqskip
  k_3(\alpha) &= \sqrt{2}D_{\frac{1}{2}(1-\alpha)}(\sqrt{2})D_{\frac{1}{2}(-1+\alpha)}(i\sqrt{2})-D_{-\frac{1}{2}(1+\alpha)}(\sqrt{2})
  [2D_{\frac{1}{2}(-1+\alpha)}(i\sqrt{2})+i\sqrt{2}D_{\frac{1}{2}(1+\alpha)}(i\sqrt{2})]\,.
\end{array}
\end{gather*}
The graph showing the dependence of the spectral determinant on the parameter~$\alpha \in (-40,2)$ is given in Figure~\ref{fig:x2:step}. Since the expression given by~\eqref{det1} is quite involved,
we have not attempted to determine its precise asymptotic behaviour. However, a (numerical) comparison with the
determinant for a Sturm-Liouville operator on a bounded interval of length two with Dirichlet boundary conditions, namely,
$2\sinh\left(2\sqrt{\alpha}\right)/\sqrt{\alpha}$ points to the behaviour for large (in absolute value)~$\alpha$
being of the same exponential type, albeit multiplied by a different power of~$\alpha$, namely, $\frac{1}{6}\sqrt {\alpha}\sinh\left (2\sqrt {\alpha} \right) $.

\begin{figure}
\centering
\includegraphics[height=5cm]{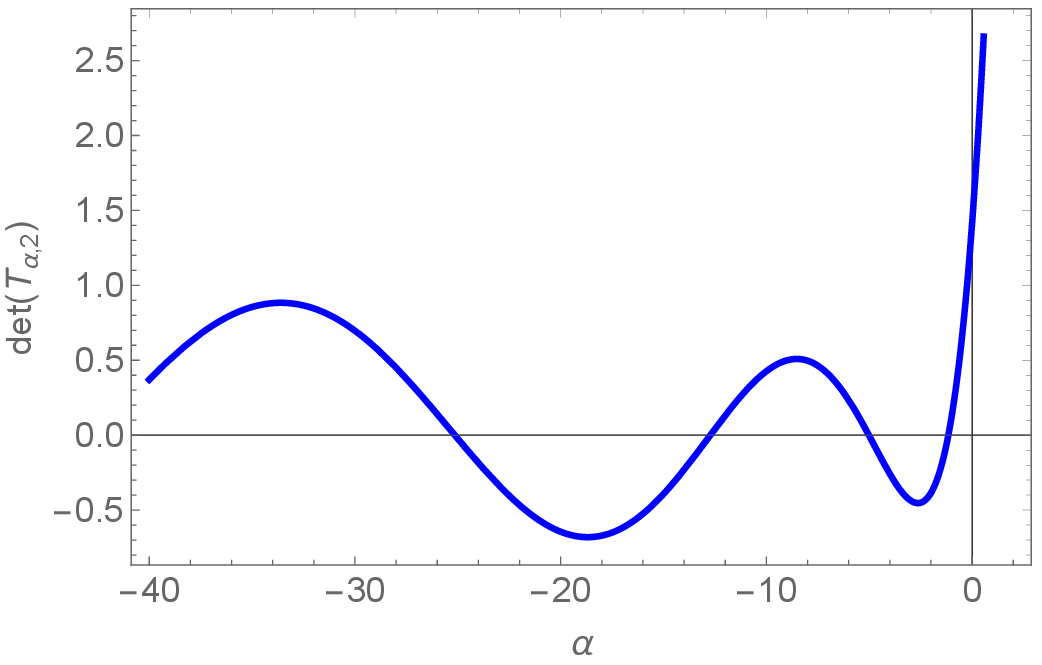}
\caption{Example~\ref{ex:x2:step}: dependence of the spectral determinant on $\alpha$.}
\label{fig:x2:step}
\end{figure}
\end{example}

\begin{example}\label{ex:x4}
Let us assume the case with $\beta=4$ and the perturbation $q(x) = x^4$ with the support $[0,1]$. Hence we have $a=0$, $b=1$. Let us first compute the Wronskian of the solutions $ y_-$ and $ y_+$ for the unperturbed case $\alpha =0$. We have from \eqref{eq:tildey1}--\eqref{eq:tildey2}
\begin{align*}
   y_- (x) &= \sqrt{-x}\, K_{\frac{1}{6}}\left(-\frac{1}{3}x^3\right)\,,\quad \alpha = 0, x<0\,,\eqskip
   y_+ (x) &= \sqrt{x} \,K_{\frac{1}{6}}\left(\frac{1}{3}x^3\right) \,,\quad \alpha = 0, x>0 
\end{align*}
and from \eqref{eq:wronskian:integer} the Wronskian of the two solutions at zero is $W(0) = 6\pi$.

For general $\alpha$ we obtain, using equations~\eqref{eq:bc3}, \eqref{eq:bc4}, \eqref{eq:integer:y+} and~\eqref{eq:integer:y+prime} and the symmetry of the problem, the conditions
\begin{align*}
   y_-(0) &= \frac{3^{1/6}}{2^{5/6}}\Gamma\left(\frac{1}{6}\right) \,,\eqskip
   y_+(1) &= K_{\frac{1}{6}}\left(\frac{1}{3}\right)\,,\eqskip
   y_-'(0) &= \frac{3^{5/6}}{2^{1/6}}\Gamma\left(\frac{5}{6}\right) \,,\eqskip
   y_+'(1) &= -K_{\frac{5}{6}}\left(\frac{1}{3}\right)\,.
\end{align*}
We construct the solution $y_{+}(x)$ in $[0,1]$ as the linear combination of the two independent solutions to the equation $-y''(x)+(1+\alpha) x^4 y(x) = 0$, namely 
$$
  \sqrt{x}\, (1+\alpha)^{\frac{1}{12}} K_{\frac{1}{6}}\left(\frac{1}{3}x^3\sqrt{1+\alpha}\right)
$$
and 
$$
  \sqrt{x}\, (1+\alpha)^{\frac{1}{12}} I_{\frac{1}{6}}\left(\frac{1}{3}x^3\sqrt{1+\alpha}\right)\,,
$$
Where $I_\nu$ denotes the modified Bessel $I$ function. Finding the coefficients of the linear combination by ``sewing'' the solutions at $x=1$, we obtain after a straightforward computation 
\begin{gather*}
\begin{array}{lll}
  y_+ (x) & = & \frac{\sqrt{x}}{3}
  \left[\sqrt{1+\alpha}K_{\frac{1}{6}}\left(\frac{1}{3}\right)\, I_{-\frac{5}{6}}\left(\frac{\sqrt{1+\alpha}}{3}\right)
  +K_{\frac{5}{6}}\left(\frac{1}{3}\right)
  I_{\frac{1}{6}}\left(\frac{\sqrt{1+\alpha}}{3}\right)\right]K_{\frac{1}{6}}\left(\frac{\sqrt{1+\alpha}}{3}x^3\right)  \eqskip
\eqskip
& & \hspace*{10mm} + \frac{\sqrt{x}}{3}\left[\sqrt{1+\alpha} K_{\frac{1}{6}}\left(\frac{1}{3}\right)
K_{\frac{5}{6}}\left(\frac{\sqrt{1+\alpha}}{3}\right)-K_{\frac{1}{6}}\left(\frac{\sqrt{1+\alpha}}{3}\right)
K_{\frac{5}{6}}\left(\frac{1}{3}\right)\right]I_{\frac{1}{6}}\left(\frac{\sqrt{1+\alpha}}{3}x^3\right).
\end{array}
\end{gather*}
The spectral determinant for $\alpha=0$ can be obtained from formula~\eqref{eq:integer:det0}
and is equal to $2$. The factor relating the functions $y_\pm$ with $v_\pm$ from Corollary~\ref{cor:integer}
is $\frac{1}{\sqrt{3\pi}}$, hence the spectral determinant is equal to the $\frac{1}{3\pi}$-multiple of the Wronskian
of the solutions $y_\pm$ at zero. A straightforward computation from Corollary~\ref{cor:integer} using Mathematica yields
\begin{align*}
  \mathrm{det}\,T_{\alpha,4} &= \frac{1}{3}\left[K_{\frac{5}{6}}\left(\frac{1}{3}\right)\left(
  \left(1+\alpha\right)^{1/12}I_{-\frac{1}{6}}\left(\frac{\sqrt{1+\alpha}}{3}\right)+
  \left(1+\alpha\right)^{-1/12}I_{\frac{1}{6}}\left(\frac{\sqrt{1+\alpha}}{3}\right)
  \right)\right.\\
  &  \hspace*{5mm}\left.+K_{\frac{1}{6}}\left(\frac{1}{3}\right)\left(
  \left(1+\alpha\right)^{5/12}I_{-\frac{5}{6}}\left(\frac{\sqrt{1+\alpha}}{3}\right)+
  \left(1+\alpha\right)^{7/12}I_{\frac{5}{6}}\left(\frac{\sqrt{1+\alpha}}{3}\right)
  \right)\right]\,.
\end{align*}
The dependence of the spectral determinant on $\alpha$ for $\alpha \in [-3000,100]$ can be found
in Figure~\ref{ex:x4:fig1}.

\begin{figure}
\centering
\includegraphics[height=5cm]{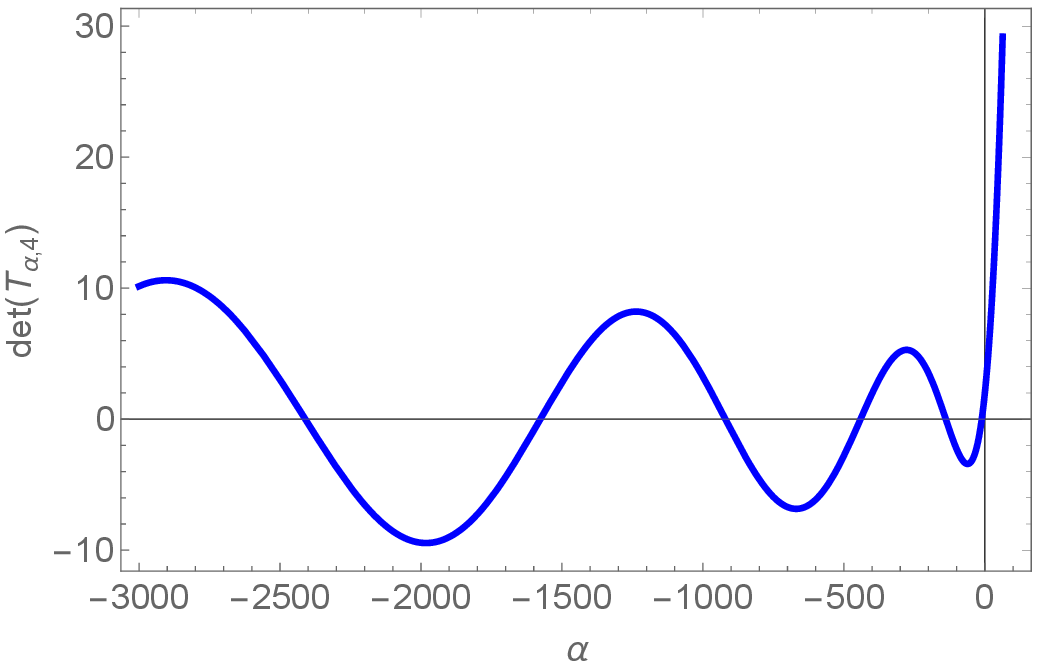}
\caption{Example~\ref{ex:x4}: dependence of spectral determinant on $\alpha$.}
\label{ex:x4:fig1}
\end{figure}
\end{example}

\section*{Acknowledgements}
P.F. was partially supported by the Funda\c c\~{a}o para a Ci\^{e}ncia e a Tecnologia (Portugal) through project UIDB/00208/2020.
J.L. was supported by the Czech Science Foundation within the project 22-18739S. We gratefully acknowledge discussions with Petr Siegl concerning conditions on potentials ensuring~\eqref{asymptotic} and~\eqref{spectr_cond}. We thank the reviewers for their valuable comments that improved the paper.

\section*{Data availability statement}
The manuscript has no associated data.

\section*{Conflict of interests}
The authors state that they have no conflicts of interest.

\end{document}